\documentclass[a4paper, 12pt]{article}
\usepackage[T1]{fontenc}
\usepackage{cmbright}

	\title{Track number of line graphs}
 
	\usepackage{authblk}

	\author{Deepak~Rajendraprasad}
 	\affil
	{
		Department of Computer Science and Engineering\\ 
		Indian Institute of Technology Palakkad
	}

	\usepackage[margin=2cm]{geometry}
	\usepackage{enumerate}
	
	\usepackage{amsmath,amsthm, amssymb}

	\newtheorem{theorem}{Theorem}[section]
	\newtheorem{corollary}[theorem]{Corollary}

	\newtheorem{conjecture}[theorem]{Conjecture}

	\theoremstyle{definition}
	\newtheorem{definition}[theorem]{Definition}

	\theoremstyle{remark}
	\newtheorem*{remark}{Remark}
	
	\def\calO{\ensuremath{\mathcal{O}}}
	\def\calF{\ensuremath{\mathcal{F}}}
	\def\calH{\ensuremath{\mathcal{H}}}

	\def\tends{\rightarrow}
	\def\into{\rightarrow}
	\def\half{\frac{1}{2}}
	\def\third{\frac{1}{3}}

	\newcommand{\ceil}[1]{\left\lceil #1 \right\rceil}
	\newcommand{\floor}[1]{\left\lfloor #1 \right\rfloor}
	\DeclareMathOperator{\elb}{elb}
	\DeclareMathOperator{\inelb}{in-elb}
	
	\DeclareMathOperator{\eq}{eq}
	\DeclareMathOperator{\cc}{cc}
	\def\track{\tau}

	\usepackage{xcolor}

\begin{document}
\maketitle

\begin{abstract}
	The {\em track number} $\tau(G)$ of a graph $G$ is 
	the minimum number of interval graphs whose union is $G$.
	We show that the track number of the line graph $L(G)$ of a 
	triangle-free graph $G$ is at least $\lg \lg \chi(G) + 1$,
	where $\chi(G)$ is the chromatic number of $G$.
	Using this lower bound and two classical Ramsey-theoretic results 
	from literature, we answer two questions posed by 
	Milans, Stolee, and West [J. Combinatorics, 2015] (MSW15).
	First we show that the track number $\track(L(K_n))$ 
	of the line graph of the complete graphs $K_n$ is at least 
	$\lg\lg n - o(1)$.
	This is asymptotically tight and it improves the bound 
	of $\Omega(\lg\lg n/ \lg\lg\lg n)$ in MSW15.
	Next we show that for a family of graphs $\mathcal{G}$,
	$\{\tau(L(G)):G \in \mathcal{G}\}$ is bounded if and only if 
	$\{\chi(G):G \in \mathcal{G}\}$ is bounded.
	This affirms a conjecture in MSW15.
	All our lower bounds apply even if one enlarges the covering family
	from the family of interval graphs to the family of chordal graphs.

\medskip
\noindent
	MSC codes: 05C55, 05C20, 05C62, 05C15.
\end{abstract}

\section{Introduction}

	The {\em track number} $\track(G)$ of a graph $G$ is 
	the minimum number of interval graphs whose union is $G$.
	Heldt, Knauer, and Ueckerdt \cite{heldt2011track} conjectured
	that the track number of line graphs is unbounded.
	Milans, Stolee, and West \cite{milans2015ordered} proved
	this conjecture by showing that the track number $\track(L(K_n))$
	of the line graph of the $n$-vertex complete graph $K_n$ is 
	$\Omega(\lg\lg n / \lg\lg\lg n)$.
	They suspected that the denominator in the lower bound could 
	be eliminated and also proposed

\begin{conjecture}[Milans, Stolee, West \cite{milans2015ordered}]
	\label{conjMilans}
	For a sequence $\left(G_n \right)_{n=1}^{\infty}$ of graphs, 
	if $\chi(G_n) \tends \infty$, then $\track(L(G_n)) \tends \infty$,
	where $\chi(G)$ and $L(G)$ denote, respectively, 
	the chromatic number and the line graph of the graph $G$.
\end{conjecture}

	In this note, first we show that $\track(L(K_n)) = (1 + o(1)) \lg\lg n$
	and then prove the above conjecture. 
 	Milans et al.\ obtain bounds on $\track(L(K_n))$ by connecting the
	problem with two problems in Ramsey theory of ordered hypergraphs.
	We use results and techniques from a paper by
	Esperet, Gimbel, and King \cite{esperet2010covering} 
	who studied the covering of line graphs with equivalence relations.
	The techniques there are close in spirit to that of 
	the Erd\"os-Szekeres theorem on total orders 
	and hence also Ramsey theoretic.
	Incidentally, the result of Esperet et al.\ disproved a 
	conjecture of McClain \cite{mcclain2009clique} that the 
	line graph of any triangle-free graph can be covered by three
	equivalence graphs. 
	We first work with triangle-free graphs and then lift the
	lower bounds obtained there to complete graphs and 
	general graphs using two classical results from Ramsey theory of graphs.

\subsection{Notation and preliminaries}
	All graphs considered in this note are finite, simple
	and do not contain self-loops. 
	Logarithm to the bases $2$  and $e$ are denoted by $\lg$ and $\ln$
	respectively.
	The line graph $L(G)$ of a graph $G$ is the intersection graph
	of the edge-set of $G$.	
	That is, two vertices of $L(G)$ are adjacent in $L(G)$
	if and only if the corresponding two edges of $G$ share
	a common vertex.
	The chromatic number of a graph $G$ is denoted by $\chi(G)$. 
	The subgraph of a graph $G$ induced on a subset $S$ of
	the vertices of $G$ is denoted by $G[S]$.

	A {\em chordal graph} is a graph with no induced cycles of 
	length more than three. 
	A graph is an {\em interval graph} if it can be represented as 
	the intersection graph of intervals on a straight line.
	An {\em equivalence graph} is a disjoint union of cliques.
	The complete graph on $n$ vertices is denoted by $K_n$.

	The {\em covering number} of a graph $G$ with respect to
	a family $\calF$ of graphs is the minimum number of graphs
	from $\calF$ whose union is $G$. 
	For example, the {\em arboricity} $a(G)$, 
	the {\em equivalence covering number} $\eq(G)$
	and the {\em track number} $\track(G)$ of a graph $G$ 
	are its covering numbers with respect to the families of 
	forests, equivalence graphs and interval graphs respectively. 
	Equivalence covering number was introduced by Duchet in 1979
	\cite{duchet1979representations} and 
	track number was introduced by Gy{\'a}rf{\'a}s and West
	in 1995 \cite{gyarfas1995multitrack}.
	For this article, we find it more natural to analyse 
	the covering number with respect to the family of chordal graphs.

\begin{definition}
\label{defnChordalCover}
	The {\em chordal covering number} $\cc(G)$ of a graph $G$ 
	is the minimum number of chordal graphs whose union is $G$.
\end{definition}

	Since equivalence graphs are interval graphs, and interval graphs are
	chordal, every graph $G$ satisfies the inequalities 
	\begin{equation}
	\label{eqnCCTrackEQ}
		\cc(G) \leq \track(G) \leq \eq(G).
	\end{equation}
	In the course of this note, it will be clear that these parameters
	are all within a factor of $2$ for line graphs of triangle-free graphs. 
	For general graphs, these parameters can be very different.
	The equivalence covering number of the $n$-vertex star graph,
	which is an interval graph,	is $n-1$. 
	As far as we have tried, we could not come up with 
	an explicit example of a chordal graph with a large track number. 
	Nevertheless we can use a counting argument to
	show that the track number of chordal graphs is unbounded. 
	Since an $n$-vertex interval graph is completely determined by
	the relative order of the $2n$ endpoints of the intervals in an
	interval representation, the number of labelled interval graphs
	on $n$ vertices is at most $(2n)!$. Hence, for any $k \geq 1$, 
	the number of labelled $n$-vertex graphs 
	which can be written as the union of $k$ interval graphs is 
	at most ${(2n)! \choose k}$ which $2^{O(kn\lg n)}$.
	On the other hand, the number of labelled $n$-vertex 
	chordal graphs is at least $2^{\Omega(n^2)}$. 
	One can see this by counting the number of labelled split graphs
	on $n$ vertices where the first $\floor{\half n}$ vertices form a clique
	and the remaining $\ceil{\half n}$ vertices can pick any subset of 
	the first $\floor{\half n}$ vertices as its neighbourhood.
	This shows that the equivalence covering number cannot be 
	bounded above by any function of the track number alone and
	the track number cannot be bounded above by any function of
	the chordal covering number alone. 

\subsection{Background}

	As mentioned earlier, we use results and techniques from
	\cite{esperet2010covering} to estimate the chordal covering number
	of line graphs.
	The starting point there is a connection that they establish 
	between equivalence coverings of $L(G)$ and 
	a certain family of orientations of $G$. 
	An {\em orientation} of an undirected simple graph $G$ is the
	directed graph formed by assigning one of the two possible
	orientations to each edge of $G$. 
	Two adjacent edges $xy$ and $xz$ of $G$ 
	are said to form an {\em elbow} in an orientation of $G$
	if both of them are directed towards $x$
	or if both of them are directed away from $x$. 
	In the first case, we will call the elbow an {\em in-elbow}
	and in the second case, we will call it an {\em out-elbow}.
	A family $\calO$ of orientations of $G$ such that 
	every pair of adjacent edges $xy$ and $xz$ in $G$
	form an in-elbow (resp., elbow)
	in at least one of the orientations in $\calO$
	is called an {\em in-elbow cover} 
	(resp., {\em elbow cover}) of $G$.
	The minimum size of an in-elbow cover (resp., elbow cover) 
	is denoted by $\inelb(G)$ (resp., $\elb(G)$).

	Esperet et al. observed that given an in-elbow cover $\calO$
	of a graph $G$, one can construct an equivalence cover
	of $L(G)$ using $|\calO|$ equivalence graphs.
	The set of vertices forming the $j$-th clique 
	in the $i$-th equivalence graph in the cover of $L(G)$,
	$1 \leq j \leq |G|$, $1 \leq i \leq |\calO|$,
	is the set of edges incident to and directed towards
	the $j$-th vertex of $G$ in the $i$-th orientation in $\calO$.
	In the other direction, they showed that,
	given an equivalence cover $\calF$ of $L(G)$,
	one can construct an in-elbow cover of $G$ using 
	$3|\calF|$ orientations of $G$. 
	Let $H$ be an equivalence subgraph of $L(G)$. 
	Every clique in $H$ corresponds to either 
	a set of edges in $G$ containing a common vertex (star-clique)
	or three edges forming a triangle in $G$ (triangle-clique).
	Consider the following three orientations of $G$ based on $H$.
	The edges of $G$ which form a star-clique in $H$ are
	oriented towards the common vertex in all the three orientations.
	The edges of $G$ which form a triangle-clique in $H$ are
	oriented such that each pair among these three edges form
	an in-elbow in one of the three orientations.
	Repeating this for every equivalence graph in an equivalence 
	cover of $L(G)$, they concluded that 
	\begin{equation}
	\label{eqnInElbEq}
		\third \inelb(G) \leq \eq(L(G)) \leq \inelb(G).
	\end{equation}
	Similarly, since the three pairs of adjacent edges in a
	triangle-clique can be elbow-covered using two orientations,
	one can also see that
	\begin{equation}
	\label{eqnElbEq}
		\half \elb(G) \leq \eq(L(G)) \leq 2 \elb(G),
	\end{equation}
	where the second inequality follows from the trivial fact that
	$\inelb(G) \leq 2 \elb(G)$.
	
	The first result in this paper is that
	$\elb(G) \leq cc(L(G))$ when $G$ is triangle-free 
	(Theorem~\ref{theoremChordalElb}).
	Before getting to it, we state and briefly discuss
	the quantitative connection between the 
	elbow covering number and the chromatic number of a graph
	that was established in \cite{esperet2010covering}.

\begin{theorem}[Theorem $10$ in \cite{esperet2010covering}]
\label{theoremElbChi}
	For any graph with at least one edge,
	\[
	\elb(G) = \ceil{\lg\lg \chi(G)} + 1.
	\]
\end{theorem}

	From Theorem~\ref{theoremElbChi} and the inequalities in 
	(\ref{eqnElbEq}), it follows that
\begin{equation} 
\label{eqnChiEq}
	\half \left(\ceil{\lg\lg \chi(G)} + 1 \right) 
		\leq eq(L(G)) 
		\leq 2 \left(\ceil{\lg\lg \chi(G)} + 1 \right). 
\end{equation} 

	They remarked towards the end of the paper that, 
	using the notion of $3$-suitability, 
	one can improve the upper bound to 
	$\lg\lg \chi(G) + \left( \half + o(1) \right) \lg\lg\lg \chi(G)$.
	A family $\calF$ of total orders of $[n]$ is \emph{$3$-suitable} if, 
	for every $3$ distinct elements $a, b, c \in [n]$ 
	there exists a total order $\sigma \in \calF$ such that 
	$a$ succeeds both $b$ and $c$ in $\sigma$ \cite{dushnik1950concerning}.
	Following Spencer \cite{spencer1972minimal}, 
	let $N(n,3)$ denote the cardinality of a 
	smallest family of total orders that is $3$-suitable for $[n]$. 
	Very tight estimates which can determine the exact value of $N(n,3)$ 
	for almost all $n$ were given by Ho\c{s}ten and Morris in 1999 
	by finding a nice equivalence of this problem to a variant of the 
	Dedekind problem \cite{hocsten1999order}. 
	It follows from there that 
	$f(n) - o(1) \leq N(n,3) \leq f(n) + 1 + o(1)$,
	where $f(n) = \lg\lg n + \half \lg\lg\lg n + \half \lg \pi$.

	Let $c : V(G) \into [k]$ be a proper vertex colouring of
	an undirected graph $G$ and 
	let $\calF$ be a family of $3$-suitable total orders of the colours $[k]$.
	For each total order in $\sigma \in \calF$ construct an orientation
	of $G$ by directing each edge $xy$ from $x$ to $y$
	if $c(x)$ precedes $c(y)$ in $\sigma$
	and the opposite otherwise.
	It is easy to verify that this family of orientations is an
	in-elbow cover of $G$.
	Hence $\inelb(G)$ and thereby $\eq(L(G))$ is at most $N(\chi(G), 3)$.

\section{Chordal covering number of line graphs}
\label{sectionProof}

	In this section we first show that, for a triangle-free graph $G$,
	the chordal covering number of $L(G)$ is at least 
	the elbow covering number of $G$.
	This lower bound can be written in terms of $\chi(G)$ using
	Theorem~\ref{theoremElbChi}.
	Using this lower bound and two classical Ramsey-theoretic results 
	from literature, we answer two questions posed by 
	Milans, Stolee, and West \cite{milans2015ordered}.

	A {\em simplicial vertex} in a graph $G$ is one whose neighbourhood
	induces a clique in $G$. A {\em perfect elimination ordering}
	of $G$ is an ordering $(v_1, \ldots, v_n)$ of $V(G)$ 
	such that $v_i$ is a simplicial vertex in $G[\{v_i, \ldots, v_n\}]$,
	for each $i$.  It is well known that a graph has a 
	perfect elimination ordering if and only if it is chordal 
	\cite{fulkerson1965incidence}.

\begin{theorem}
\label{theoremChordalElb}
For every triangle-free graph $G$,
\[
	\elb(G) \leq \cc(L(G)).
\]
\end{theorem}

\begin{proof}
	Let $G$ be any triangle-free graph. 
	Let $\calH$ be a smallest collection of chordal graphs 
	whose union is $L(G)$.
	Based on each chordal graph $H \in \calH$, we construct an orientation
	$O_H$ of $G$ such that every pair of edges of $G$
	which are adjacent as vertices in $H$ will form an elbow in $O_H$. 
	Since every pair of adjacent edges of $G$
	are adjacent as vertices in at least one $H$ in $\calH$, 
	it is easy to verify that the family of $|\calH|$ orientations
	constructed with the promised property will serve as an
	elbow-cover of $G$ with size $\cc(L(G))$. 

	Let $H \in \calH$ be arbitrary.
	By allowing isolated vertices if necessary, 
	we assume that $H$ is a spanning subgraph of $L(G)$.
	Consider a perfect elimination ordering $e_1, \ldots, e_m$ of $H$,
	where $m$ is the number of edges in $G$. 
	That is, $\forall i \in [m]$, $e_i$ is a simplicial vertex in 
	$H_i = H[\{e_i, \ldots, e_m\}]$. 
	In order to keep the notation clean, we will (ab)use 
	the same name for a vertex of $H$ and the corresponding edge in $G$.
	For each $i$ going from $m$ down to $1$, the edge $e_i$ in $G$
	is oriented so that it forms an elbow in $O_H$ with the 
	most recently oriented edge of $G$ 
	which is adjacent as a vertex to $e_i$ in $H_i$.
	If $e_i$ has no neighbours in $H_i$, then it is oriented arbitrarily.

	Now we argue that every pair of edges in $G$ 
	which are adjacent as vertices in $H$ 
	will be oriented to form an elbow in $O_H$.
	For each $i \in [m]$, let $N_i$ denote the neighbours of $e_i$ in $H_i$. 
	We call the orientation of an edge $e_i$ in $O_H$ ``good'' 
	if it forms an elbow in $O_H$ with every edge of $G$ which corresponds
	to a vertex in $N_i$. 
	It is enough to show that every edge $e_i, i \in [m]$	is good.  
	We show this by induction on $(m-i)$.
	Vacuously, $e_m$ is good. 
	For some $i < m$, let us assume, by induction, that
	$e_{i'}$ is good for all $i' > i$.
	If $|N_i| \leq 1$, then it is clear that $e_i$ will be oriented good. 
	If $|N_i| \geq 2$, let $j = \min\{k : e_k \in N_i\}$.
	By construction $e_i$ and $e_j$ form an elbow in $O_H$.
	Moreover, $e_j$ is oriented good in $H_j$ and hence
	$e_j$ forms an elbow with every edge corresponding to a vertex in $N_j$.
	Since $e_i$ is simplicial, $N_i \cup \{e_i\}$ induces a clique in $H_i$;
	that is, the corresponding edges are pairwise adjacent in $G$.
	Since $G$ is triangle-free, these edges share a common vertex.  
	Since $e_i$ forms an elbow with $e_j$ and $e_j$
	forms an elbow with every edge corresponding to a vertex in 
	$N_j$, which is a superset of $N_i \setminus \{e_j\}$, 
	we see that $e_i$ is also oriented good.
\end{proof}

\begin{remark}
	From Theorem~\ref{theoremElbChi}, (\ref{eqnCCTrackEQ}) and (\ref{eqnElbEq}), 
	we see that for every triangle-free graph $G$,
	\[
		\cc(L(G)) \leq \track(L(G)) \leq \eq(L(G)) \leq 2\cc(L(G)).
	\]
\end{remark}

	From Theorem~\ref{theoremElbChi} and Theorem~\ref{theoremChordalElb}
	one can immediately infer

\begin{corollary}
\label{corTriangleFreeLB}
For every triangle-free graph $G$,
\[
	\ceil{\lg\lg \chi(G)} + 1 \leq \cc(L(G)).
\]
\end{corollary}

	We can use the above result together with some celebrated 
	Ramsey-theoretic results to estimate the chordal covering number
	of complete graphs and general graphs.  
	Since the family of chordal graphs is hereditary, 
	$\cc(G') \leq \cc(G)$ whenever $G'$ is an induced subgraph of a graph $G$.
	Since the line graph of a subgraph is an 
	induced subgraph of the line graph of the original graph,
	$\cc(L(H')) \leq \cc(L(H))$ whenever $H'$ is a subgraph of $H$.

	It was established by 
	Kim \cite{kim1995ramsey} that for every sufficiently large $n$,
	there exists an $n$-vertex triangle-free graph $G_n$ with 
	\[
		\chi(G_n)  \geq \frac{1}{9} \sqrt{\frac{n}{\ln n}}.
	\]
	Since $G_n$ is a subgraph of $K_n$, $\cc(L(G_n)) \leq \cc(L(K_n))$.
	This gives the lower bound in

\begin{corollary}
\label{corCompleteLB}
\[
	\lg\lg n - o(1) \leq \cc(L(K_n)) \leq 
		\lg\lg n + \half \lg\lg\lg n + \half \lg \pi + 1 + o(1).
\]
\end{corollary}
	The upper bound follows from the inequality $\eq(L(K_n)) \leq N(n,3)$.
	So we can remove the denominator from the lower bound of 
	$\Omega(\lg\lg n / \lg\lg\lg n)$  on $\track(L(K_n))$
    from \cite{milans2015ordered} as suspected by the authors. 
	Furthermore, it shows that $\track(L(K_n))$ is
	asymptotically $(1+o(1))\lg\lg n$.

	Finally, we use these two results together with 
	a beautiful result of R\"odl to prove 
	Conjecture~\ref{conjMilans}. 
	It was shown by R\"odl \cite{rodl1977chromatic} that, 
	for arbitrary positive integers $m$ and $n$,
	there exits a $\phi(m,n)$ such that if $\chi(G) \geq \phi(m,n)$, then
	the graph $G$ contains either a clique of size $m$ or a
	triangle-free subgraph $H$ with $\chi(H) = n$.
	Consider any sequence $\left(G_n \right)_{n=1}^{\infty}$ of graphs, 
	with $\chi(G_n) \tends \infty$. 
	Suppose $\track(L(G_n))$ was bounded above by some constant $b$.
	let $B = 2^{2^{b+1}}$ and choose a graph $G$ from the sequence $(G_n)$
	with $\chi(G) \geq \phi(B,B)$.
	Using R\"odl's result, we can conclude that $G$ 
	either contains a  $B$-vertex complete graph $K_B$ or 
	a triangle-free graph $H$ with $\chi(H) = B$.
	In either case, we have shown that the chordal covering number 
	of the line graph that subgraph is more than $b$ 
	(Corollaries \ref{corCompleteLB} and \ref{corTriangleFreeLB}).
	This contradiction proves 

\begin{theorem}
	\label{theoremUnbounded}
	For a sequence $\left(G_n \right)_{n=1}^{\infty}$ of graphs, 
	if $\chi(G_n) \tends \infty$, then $\cc(L(G_n)) \tends \infty$.
\end{theorem}

	Thus we affirm Conjecture~\ref{conjMilans}. 
	Further, since $\cc(L(G)) \leq \eq(L(G)) \leq N(\chi(G), 3)$,
	we see that, for a family of graphs $\mathcal{G}$,
	$\{\tau(L(G)):G \in \mathcal{G}\}$ is bounded if and only if 
	$\{\chi(G):G \in \mathcal{G}\}$ is bounded.

\section{Concluding remarks}
	The function $\phi(m,n)$ obtained by R\"odl 
	is a tower of $n$'s of height $m$. 
	Hence the lower bound obtained for $\track(L(G))$ for a general
	graph $G$ in terms of $\chi(G)$ is of very small order. 
	We suspect that, like $\eq(L(G))$, $\cc(L(G))$ might also be
	bounded below by $\Omega(\lg\lg \chi(G))$.

\bibliographystyle{alpha}

\end{document}